\pdfoutput=1
\documentclass[reqno,12pt]{amsart}
\usepackage{float}
\usepackage[mathscr]{eucal}
\usepackage{amscd}
\usepackage{amsfonts}
\usepackage{booktabs}
\usepackage{multirow}
\usepackage{amsmath, enumerate, amsthm, amssymb}
\usepackage{latexsym}
\usepackage[pdftex]{graphics}
\usepackage{graphicx}
\usepackage[font=small,skip=0pt]{caption}
\usepackage{float}
\theoremstyle{plain}
\newtheorem{thm}{Theorem}[section]

\newtheorem{lem}[thm]{Lemma}  
\newtheorem{prop}{Proposition}[section]
\theoremstyle{definition}
\newtheorem{defn}{Definition}[section]

\theoremstyle{remark}

\allowdisplaybreaks
\newcommand{\NaturalNumber}{\mathbb N}
\newcommand{\R}{\mathbb{R}}

\numberwithin{equation}{section}
\usepackage[colorlinks,linkcolor=blue,citecolor=blue,urlcolor=blue,bookmarks=false,hypertexnames=true]{hyperref}
\begin{document}
\title[Fixed Point Theory]
	{A simpler and more efficient fixed point iterative scheme}
	\author[Nida Izhar Mallick, Izhar Uddin]{ Nida Izhar Mallick$^{1}$, Izhar Uddin$^{2, *}$ }
	\maketitle
	\begin{center}
		{\footnotesize  $^{1,2}$Department of Mathematics, Jamia Millia Islamia,\\ New Delhi-110025, India.\\
			nidamallick2016@gmail.com, izharuddin1@jmi.ac.in.
		}
	\end{center}
	{\footnotesize \noindent{\bf Abstract.} Our work presents a new iterative scheme to approximate the fixed points of nonexpansive mapping. The proposed algorithm is constructed to enhance convergence efficiency while preserving theoretical robustness. Under appropriate assumptions on the underlying operator, we establish weak convergence and strong convergence results for the generated sequence. To demonstrate the effectiveness of the proposed scheme, we present a numerical example and perform a detailed comparative study with several well-known iterative methods from the literature. The numerical results clearly indicate that the proposed method exhibits a faster rate of convergence than the existing schemes, thereby confirming its computational advantage. These findings suggest that the new iterative process provides an efficient and reliable alternative for solving fixed point problems arising in applied mathematics and related fields.  \vskip0.5cm \noindent {\bf Keywords}:
		Nonexpansive mappings, Fixed point theory.
		
		}
		\section{\textbf{Introduction}}
         Fixed point theory is an essential area of mathematical analysis that studies the existence and characteristics of points that remain unchanged under a specific mapping. Formally, if $W$ is a set and $K: W \to W$ is a mapping, a point $w \in W$ is called a fixed point if $K(w)=w$. Though it appears to be a simple concept, it holds broad and lasting significance in both pure and applied mathematics, serving as a powerful tool for establishing the uniqueness and existence of solutions to a wide range of mathematical problems. This significance is reflected in its diverse applications across numerous disciplines: in analysis and differential equations, fixed point results are frequently employed to establish the existence of solutions to nonlinear integral, functional, and differential equations \cite{AgElGeOr2008,KhKi2001,GoKi1990,SaAlAlUd2024, AlSaAkUd2024}; in optimization theory, they provide the theoretical foundation for the convergence of iterative algorithms \cite{SaBe2012,Xu2002}; in engineering and the applied sciences, fixed point techniques appear in control theory \cite{De1985}, signal processing \cite{Ce2012}, and image restoration processes \cite{By2002,CeElKoBo2005}.\\
         The origin of fixed point theory traces back to classical results such as Brouwer’s and Banach’s fixed point theorems, which laid the foundation for modern nonlinear analysis and inspired an extensive range of generalizations and applications. Its importance lies not only in its elegant theoretical structure but also in its ability to unify diverse problems under a single conceptual framework.\\
         Beyond proving the existence of fixed points, a major research direction concerns their approximation. Iterative methods provide practical means to compute fixed points even in the absence of exact analytical solutions. Over time, numerous fixed point theorems with more efficient iterative methods have been introduced and studied, see for example \cite{Abb, Aga, Ish, Kha, Man, Noo, Tha, M*, M, K, K*}. This progression from foundational theorems to advanced iterative algorithms reflects the field’s continuous evolution, where classical principles serve as the basis for modern computational techniques adapted to increasingly complex and high-dimensional problems.\\
         Among the earliest and most influential iterative methods are the Mann\cite{Man}, Ishikawa\cite{Ish}, and Noor\cite{Noo} schemes, which have played a fundamental role in shaping the theory and practice of fixed point approximation. These classical algorithms have inspired a range of modifications and extensions aimed at improving convergence speed, stability, and applicability to broader classes of mappings. One may refer to the following references \cite{Art, Khu}, as these basic iterative schemes find many applications across different branches of mathematics and applied sciences. Motivated by the strengths of these classical iterative schemes, we now propose a new iterative algorithm, defined as follows:\\
         Let $Z$ be a Banach space and $W$  is its  nonempty closed and convex subset and $K: W \to W$ be a self-nonexpansive mapping. Then,
		\begin{equation}\label{eq.2}
			\begin{cases}
				s_1 \in W \\
				r_n=  (1-\beta_n)s_n+\beta_n Ks_n,\\
				s_{n+1}= (1-\alpha_n)r_n+\alpha_nKs_n.
			\end{cases}
		\end{equation}
		with $\alpha_n, \beta_n \in (a, b) \subset (0, 1)$ for all $n \in \mathbf{N}$, where $a, b \in (0, 1)$.\\
        Unlike the Mann iteration, which employs a single weighted average at each step, the proposed algorithm (\ref{eq.2}) incorporates an intermediate step similar in spirit to Ishikawa’s method but with a distinct weighting structure, allowing for more flexible parameter selection. Compared to Noor’s three-step approach, our method achieves a balance between computational simplicity and enhanced convergence behavior, making it particularly suitable for certain classes of self-nonexpansive operators where both speed and stability are critical.To further clarify the structural differences, Figure [\ref{fig:1}, \ref{fig:2}] presents a schematic comparison between the Ishikawa iteration and the proposed algorithm. The diagram emphasizes how the intermediate step is positioned, resulting in a modified update sequence that preserves simplicity while potentially improving convergence performance.

        \begin{figure}[H]
        	\centering
        	\hspace*{0.2cm}
        	\includegraphics[width=1.02\textwidth]{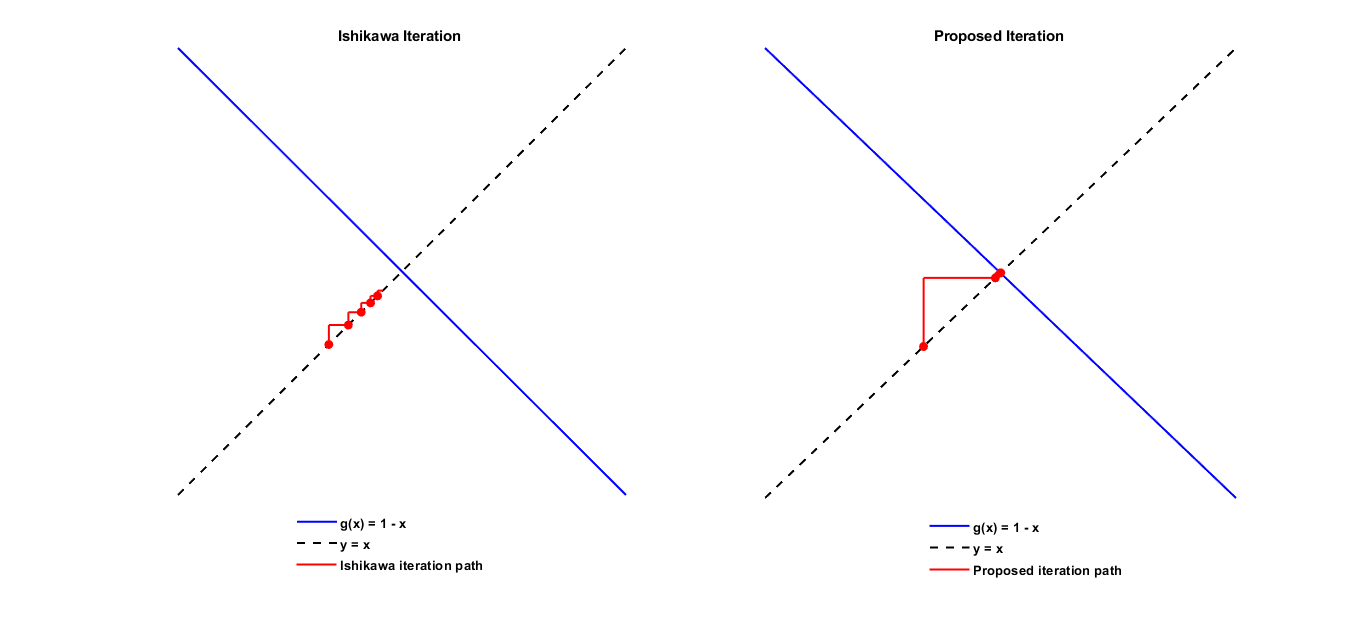} 
        	\vspace{-30pt}
            \centering
        	\caption{Convergence path of Ishikawa iteration and the proposed iteration for the fixed point of $g(x)=1-x$.}
        	\label{fig:1}
        \end{figure}

        \begin{figure}[H]
        	\centering
        	\hspace*{0.2cm}
        	\includegraphics[width=1.05\textwidth]{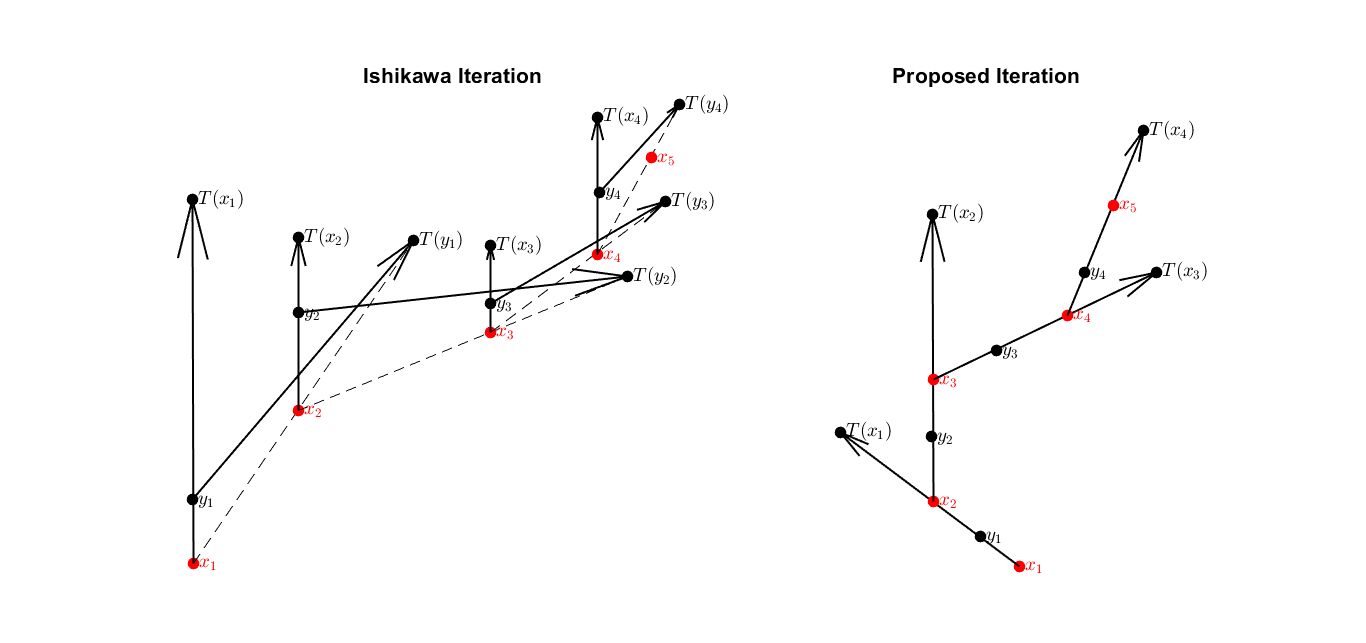} 
        	\vspace{-10pt}
            \centering
        	\caption{Difference between Ishikawa iteration and our proposed iteration.}
        	\label{fig:2}
       \end{figure}

         Figure [\ref{fig:1}] illustrates that the Ishikawa process requires several steps to approach the fixed point, whereas the proposed scheme reaches it in fewer iterations, demonstrating its faster convergence behavior. Similarly, figure [\ref{fig:2}] highlights that the proposed scheme reaches successive approximations more directly and efficiently, requiring fewer intermediate steps than the Ishikawa process.  
		\section{\textbf{Preliminaries}}
		\label{SC:2.Preliminaries}
		In this section, we summarize the basic definitions, lemmas, and propositions that will be applied to demonstrate the convergence theorems.
		\begin{defn}\label{def.3}\cite{Clar}
			A Banach space $Z$ is uniformly convex if for each $\epsilon \in (0,2]$ there exists a $\delta>0$ such that for $s,t \in Z$ with $\|s\|\leq 1$, $\|t\|\leq 1$ and $\|s-t\|\geq \epsilon$, it holds that $$\|\frac{s+t}{2}\|<1-\delta.$$
		\end{defn}
		\begin{defn}\label{def.4}\cite{Op}
			A Banach space $Z$ is fulfills Opial's condition whenever any sequence $\{s_n\}$ in $Z$ converges weakly to $s \in Z$ i.e., $s_n \rightharpoonup s$ it follows that $$\limsup_{n\to\infty}\|s_n-s\|<\limsup_{n\to\infty}\|s_n-t\|$$
			for every $t \in Z$ with $s \neq t$.\\
            A mapping $K: W \to Z$ is said to be demiclosed at $y \in Z$ if for each sequence $\{x_n\}$ in $W$ and each $x \in Z, x_n \rightharpoonup x$ and $Kx_n \to y $ imply that $x \in W$ and $Kx=y$.\\
			Let $W \subset Z$ be a nonempty, closed and convex, and suppose that $\{s_n\}$ is a bounded sequence in $Z$.
			\\For each $s \in Z$ write:
			$$r(s,\{s_n\})=\limsup_{n\to\infty}\|s-s_n\|.$$
			The value of asymptotic radius of the sequence $\{s_n\}$ with respect to $W$ is determined by
			$$r(W,\{s_n\})=inf\{r(s,\{s_n\}): s \in W\},$$
			and the value of asymptotic center $A(W,\{s_n\})$ of the sequence $\{s_n\}$ is determined by:
			$$A(W,\{s_n\})=\{s \in W: r(s,\{s_n\})=r(W,\{s_n\}) \}.$$
			A widely known result states that $A(W,\{s_n\})$ contains atleast one element if $W$ is weakly compact. Moreover, if there is a closed and convex subset $W$ of a uniformly convex Banach space $Z$ such that the sequence $\{s_n\}$ is bounded in $W$, then the asymptotic center uniquely determined \cite{Ed1,Ed2}.
		\end{defn}
        \begin{lem}\label{def.5a}\cite{G}
            Let $Z$ be a uniformly convex Banach space, $W$ a closed and convex subset of $Z$, and $K: W \to W$ nonexpansive mapping. Then $I-T$ is demiclosed on $W$.
        \end{lem}
		\begin{prop}\label{def.5}\cite{Sc}
			Let $Z$ to be a uniformly convex Banach space and $\{q_n\}$ be any sequence such that $0<p \leq q_n \leq r<1$ for all $n \geq 1$ and for some $p,r \in \R$. Let there be any two seuences $\{s_n\}$ and $\{t_n\}$ of $Z$ such that $\limsup_{n\to\infty}\|s_n\| \leq g$, $\limsup_{n\to\infty}\|t_n\| \leq g$ and $\lim_{n\to\infty}\|q_ns_n+(1-q_n)t_n\|=g$ for some $g \geq 0$. Then, $\lim_{n\to\infty}\|s_n-t_n\|=0$.
		\end{prop}
		\begin{defn}\label{def.6}
			\cite{Se} A self mapping $K$ on a nonempty subset $W$ of a normed space adheres to Condition $(I)$ if there exists a function $m:[0, \infty) \to [0, \infty)$ satisfying $m(0)=0$ and $m(i)>0$ for all $i \in (0, \infty)$ which is nondecreasing such that $\|s-Ks\| \geq m(d(s, f(K)))$ for all $s\in W$. Here, $F(K)$ is the set of all fixed points of $K$ and $d(s, f(K)=inf\{\|s-t\|: t \in F(K)\}$.
		\end{defn}
		\begin{defn}\label{def.7}
			For a mapping $K: W \to W$ if a sequence $\{s_n\}$ in $W$ satisfies $\lim_{n\to\infty}\|Ks_n-s_n\|=0$, then this sequence $\{s_n\}$ is known as approximate fixed point sequence(a.f.p.s. for short).
		\end{defn}
        \begin{thm}\label{thm.a}
        \cite{Xu} Let $p>1$, $r>0$ be two fixed real numbers. Then a Banach space $Z$ is said to be uniformly convex if and only if there is a continuous, strictly increasing and convex function $g: [0. \infty] \rightarrow [0. \infty], g(0)=0,$ such that 
        $$ \|\lambda s + (1-\lambda)t\|^p \leq \lambda\|s\|^p + (1-\lambda)\|t\|^p - W_p(\lambda)g(\|s-t\|)$$
        for all $s, t$ in $B_r$ and $0 \leq \lambda \leq 1$. Here $W_p=\lambda(1-\lambda)^p + \lambda^p(1-\lambda)$ and $B_r$ is a closed ball with center zero and radius $r$.
        \end{thm}
		
		\section{\textbf{Main Result}}
		\label{SC:2.Main Result}
		  In this section, we establish the strong and weak convergence theorems for the sequence produced by our proposed iteration (1.1) for nonexpansive mappings in Banach spaces.\\
 	\begin{lem}\label{lem.6}\ Let $W \neq\emptyset$ be a closed convex subset of a Banach space $Z$ and $K:W \rightarrow W$ be a nonexpansive mapping with $F(K)\neq\emptyset$. If sequence $\{s_n\}$ be defined as in $(\ref{eq.2})$ then $ \lim_{n\to\infty}||s_n-w||$ exists for any $ w \in F(K)$.
		\end{lem}
		\begin{proof} Taking $w\in F(K)$ and $v\in W$. Since, $K$ is a nonexpansive mapping we have,
			$\|Kv-Kw\| \leq \|v-w\|$.
			\\ Now,\\
			\begin{align}\label{eq.13}
				\|t_n-w\|&=\|\{(1-\beta_n)s_n+\beta_nKs_n\}-w\|\nonumber\\
				& \leq(1-\beta_n)\|s_n-w\|+\beta_n\|Ks_n-w\|\nonumber\\
				& \leq(1-\beta_n)\|s_n-w\|+\beta_n\|s_n-w\|\nonumber\\
				&= \|s_n-w\|
			\end{align}
			By using \eqref{eq.13}, we have
			\begin{align}\label{eq.14}
				\|s_{n+1}-w\|&=\|\{(1-\alpha_n)t_n+\alpha_nKs_n\}-w\|\nonumber\\
				& \leq (1-\alpha_n)\|t_n-w\|+\alpha_n\|Ks_n-w\|\nonumber\\
			    & \leq (1-\alpha_n)\|t_n-w\|+\alpha_n\|s_n-w\|\nonumber\\
			    & \leq (1-\alpha_n)\|s_n-w\|+\alpha_n\|Ks_n-w\|\nonumber\\
				&= \|s_n-w\|.\nonumber\\
			\end{align}
				This implies that ${\|s_n-w\|}$ is a decreasing sequence which is bounded below for all $w \in F(K)$.
			Hence, $\lim_{n\to\infty}\|s_n-w\|$ exists.
		    \end{proof}
			\begin{lem}\label{lem.7}\ Let there be a nonempty subset $W$ of a uniformly convex Banach space $Z$ which is closed and convex  and $K:W \rightarrow W$ is a self nonexpansive mapping having nonempty $F(K)$. Let $\{s_n\}$ be a sequence with $\{s_1\}\in W$ defined as in $(\ref{eq.2})$ with $\{\alpha_n\}, \{\beta_n\} \in (a, b) \subset (0, 1)$ for all $n \in \mathbf{N}$, where $a, b \in (0, 1)$.\\. Then $F(K)\neq\emptyset$ if and only if $\{s_n\}$ is bounded and $\lim_{n\to\infty}\|s_n-Ks_n\|=0$.
			\end{lem}
			\begin{proof} Consider $F(K)\neq\emptyset$ and take $w \in F(K)$. Then, lemma (\ref{lem.6}) provides the existence of $\lim_{n\to\infty}\|s_n-w\|$ and boundedness of $\{s_n\}$. Let
				$$\lim_{n\to\infty}\|s_n-w\|=r.$$
			From \eqref{eq.13}, we have
			$$\limsup_{n\to\infty}\|t_n-w\| \leq \limsup_{n\to\infty}\|s_n-w\|=r.$$
			Since, $K$ is nonexpansive mapping, we have\\
            $$\|Ks_n-w\| \leq \|s_n-w\|=r \hspace{0.5cm}and \hspace{0.5cm}\|Kt_n-w\| \leq \|t_n-w\|=r.$$
            Taking limsup on both sides, we obtain
            $$\limsup_{n\to\infty}\|Ks_n-w\| \leq r$$
            $$\limsup_{n\to\infty}\|Kt_n-w\| \leq r$$
            Also,
            \begin{align}\label{eq.15}
			 r&=\lim_{n\to\infty}\|s_{n+1}-w\|\nonumber\\
               &=\lim_{n\to\infty}\|\{(1-\alpha_n)t_n+\alpha_nKs_n\}-w\|\nonumber\\
              &=\lim_{n\to\infty}\|(1-\alpha_n)(t_n-w)+\alpha_n(Ks_n-w)\|\nonumber\\
			\end{align}
            Using proposition (\ref{def.5}), we have
            $$\lim_{n\to\infty}\|t_n-Ks_n\|=0.$$
            Now,
            \begin{align}\label{eq.15}
			 \|s_{n+1}-w\|
               &=\|\{(1-\alpha_n)t_n+\alpha_nKs_n\}-w\|\nonumber\\
              &=\|t_n-\alpha_nt_n +\alpha_nKs_n-w\|\nonumber\\
              &=\|(t_n-w)-\alpha_n(t_n +Ks_n)\|\nonumber\\
              &\leq \|t_n-w\| + \alpha_n\|t_n-Ks_n\|
			\end{align}
            $$ \implies r\leq \liminf_{n\to\infty}\|t_n-w\|$$ 
            we get
             $$\lim_{n\to\infty}\|t_n-w\|= r.$$
             Now, using Theorem (\ref{thm.a}), we have
              \begin{align}\label{eq.16}
			 \|t_n-w\|^2
               &=\|\{(1-\beta_n)s_n+\beta_nKs_n\}-w\|^2\nonumber\\
              &\leq(1-\beta_n)\|s_n-w\|^2 + \beta_n\|Ks_n-w\|^2 - \beta_n(1-\beta_n)\|s_n-Ks_n\|^2\nonumber\\
              &\leq\|s_n-w\|^2-\beta_n(1-\beta_n)\|s_n-w\|^2\nonumber\\
			\end{align}
            so that,
            $$\beta_n(1-\beta_n)\|s_n-Ks_n\|^2 \leq \|s_n-w\|^2 - \|t_n-w\|^2.$$
            Also, we have
            \begin{align}\label{eq.17}
             \|s_n-Ks_n\|^2 &\leq \frac{1}{\beta_n(1-\beta_n)} [\|s_n-w\|^2 - \|t_n-w\|^2] \nonumber\\
             &\leq \frac{1}{a(1-b)} [\|s_n-w\|^2 - \|t_n-w\|^2]. \nonumber\\
            \end{align}
            By taking $\lim_{n\to\infty}$ on both sides, we get 
            $$\|s_n-Ks_n\|=0.$$
            \end{proof}
		\begin{thm}\label{thm.1}
			Let $K$, $W$ and $Z$ be same as defined in Lemma  \ref{lem.7} such that the mapping  $K$ satisfies the Opial's condition with $F(K)\neq\emptyset$. If $\{s_n\}$ is the iterative sequence defined by iterative process $(\ref{eq.2})$.Then, $\{s_n\}$  converges  weakly  to  a  fixed  point  of $K$.
		\end{thm}
		\begin{proof}
			Let $w \in F(K)$. Then, from Lemma \ref{lem.6} $\lim_{n\to\infty}\|s_n-w\|$ exists. In order to show the weak convergence of the iteration (\ref{eq.2}) to a unique fixed point of $K$, we will verify that the sequence $\{s_n\}$ contains a unique weak subsequential limit in $F(K)$. For this, consider two subsequences $\{s_{n_j}\}$ and $\{s_{n_k}\}$ of $\{s_n\}$ which weakly converges to $u$ and $v$ respectively. From Lemma \ref{lem.7}, we have $\lim_{n\to\infty}\|Ks_n-s_n\|=0$ and from Lemma \ref{def.5a}, it is evident that $(I-K)$ is demiclosed at zero which results in $u,v \in F(K)$.\\
			Now, we proceed to verify the uniqueness. As $u,v \in F(K)$, it means $\lim_{n\to\infty}\|s_n-u\|$ and $\lim_{n\to\infty}\|s_n-v\|$ exists. Let us assume $u \neq v$.\\
			Then, applying Opial's condition, we obtain
			\begin{align}
				\lim_{n\to\infty}\|s_n-u\|=\lim_{j\to\infty}\|s_{n_j}-u\|\nonumber\\
				<\lim_{j\to\infty}\|s_{n_j}-v\|\nonumber\\
				=\lim_{n\to\infty}\|s_n-v\|\nonumber\\
				=\lim_{k\to\infty}\|s_{n_k}-v\|\nonumber\\
				<\lim_{k\to\infty}\|s_{n_k}-u\|\nonumber\\	
				=\lim_{n\to\infty}\|s_n-u\|.\nonumber
			\end{align}
			Which contradicts our assumption, so $u=v$. Thus, the weak convergence of the sequence $\{s_n\}$ to a fixed point of $K$ is established.
		\end{proof}
		\begin{thm}\label{thm.2}
			Let $K$ be a nonexpansive mapping defined on a nonempty closed convex subset $W$ of a uniformly convex Banach space $Z$ with $F(K)\neq\emptyset$. If $\{s_n\}$ is the iterative sequence defined by iterative process $(\ref{eq.2})$,Then $\{s_n\}$  converges strongly to a point of $F(K)$ if and only if $\liminf_{n\to\infty}d(s_n,F(K))=0$.
		\end{thm}
		\begin{proof}
			If the sequence $\{s_n\}$ converges to a point $w\in F(K)$, then it is obvious that $\liminf_{n\to\infty}d(s_n,F(K))=0$. 
			For the converse part, assume that $\liminf_{n\to\infty}d(s_n,F(K))=0$. From Lemma \ref{lem.6}, we have $\lim_{n\to\infty}\|s_n-w\|$ exists for all $w\in F(K)$, which gives
			\begin{align}
				\|s_{n+1}-w\|\leq \|s_n-w\|  \hspace{0.5cm}   for\hspace{0.1cm} any \hspace{0.1cm} w \in F(K) \nonumber
			\end{align}
				which yields,
			\begin{align}
				d(s_{n+1},F(K))\leq d(s_n,F(K))   \nonumber 
			\end{align}
			Thus, $\{d(s_n,F(K))\}$ is a non-decreasing sequence with zero as its lower bounded as well, this establishes that $\lim_{n\to\infty}d(s_n,F(K))$ exists. As, $\liminf_{n\to\infty}d(s_n,F(K))=0$ so $\lim_{n\to\infty}d(s_n,F(K))=0$.
			Now, we find a subsequence $\{s_{n_j}\}$ of $\{s_n\}$ and a sequence $\{w_j\}$ in $F(K)$ such that $\|s_{n_j}-w_j\|\leq\frac{1}{2^j}$ for all $j \in \NaturalNumber$. From the proof of Lemma \ref{lem.6}, we have  
			\begin{align}
				\|s_{n_{j+1}}-w_j\|\leq \|s_{n_j}-w_j\|\leq\frac{1}{2^j}   \nonumber 
			\end{align}
				Using triangle inequality, we get
			\begin{align*}
				\|w_{j+1}-w_j\|&\leq\|w_{j+1}-s_{j+1}\|+\|s_{j+1}-w_j\|\nonumber\\
				&<\frac{1}{2^{j+1}} + \frac{1}{2^j}\nonumber\\
				&\leq\frac{1}{2^{j-1}}\nonumber\\
				&\to 0 \hspace{0.5cm} as\hspace{0.2cm} j \to \infty
			\end{align*}
			This shows that $\{w_j\}$ is a cauchy sequence in $F(K)$. It is given that $F(K)$ is closed implies that $\{w_j\}$ converges to some $w \in F(K)$.
			Again, using triangle inequality, we have
			\begin{align}
				\|s_{n_j}-w\|\leq \|s_{n_j}-w_j\|+\|w_j-w\|   \nonumber 
			\end{align}
			As $j \to \infty$, we have $\{s_{n_j}\}$ converges to $w \in F(K)$ strongly. As from Lemma 3.1, $\lim_{n\to\infty}\|s_n-w\|$ exists for all $w \in F(K)$ implies that $\{s_n\}$ converges strongly to $w \in F(K)$.
		\end{proof}
		Now, we prove the following strong convergence theorem for the mapping satisfying Condition $(I)$.
		\begin{thm}\label{thm.3}
			Let $K$ be a nonexpansive mapping defined on a nonempty closed convex subset $W$ of a uniformly convex Banach space $Z$ such that $F(K)\neq\emptyset$ and $\{s_n\}$ be the sequence defined by $(\ref{eq.2})$. If $K$ satifies Condition $(I)$, then $\{s_n\}$ converges strongly to a fixed point of $K$.
		\end{thm}
		\begin{proof}
			By Lemma 3.1, $\lim_{n\to\infty}\|s_n-w\|$ exists for all $w \in F(K)$ and so $\lim_{n\to\infty}\|s_n-F(K)\|$ exists. Assume that $\lim_{n\to\infty}\|s_n-w\|=r$ for some $r \geq 0$. If $r=0$ then the result follows.\\
			Suppose, $r>0$. From the hypothesis and condition $(I)$,
			\begin{align}
				m(d(s_n,F(K)))\leq \|s_n-Ks_n\|.\nonumber
			\end{align}
			Since, $F(K)\neq \emptyset$, by Lemma 3.2, we have $\lim_{n\to\infty}\|s_n-Ks_n\|=0$, which implies
			\begin{align}
				\liminf_{n\to\infty}m(d(s_n,F(K)))=0.\nonumber
			\end{align}
			Since, $m$ is a nondecreasing function then we have $\lim_{n\to\infty}d(s_n,F(K))=0$. Thus, we have a subsequence $\{s_{n_k}\}$ of $\{s_n\}$ and a subsequence $\{t_n\} \subset F(K)$ such that $$d(s_{n_k},t_{n_k}) < \frac{1}{2^k},$$
			for all $k \in \mathbb{N}$. Hence,
			$$d(t_{k+1},t_k) \leq d(t_{k+1},s_{k+1})+d(s_{k+1},t_k) \leq \frac{1}{2^{k-1}} \to 0,$$
			as $k \to \infty$. This shows that $\{t_k\}$ is a Cauchy sequence in $F(K)$ and so it converges to a point $w$. Since, $F(K)$ is closed and $w \in F(K)$ then $\{s_{n_k}\}$ converges strongly to $w$. Since, $\lim_{n\to\infty}\|s_n-w\|$ exists, we have $\{s_n\} \to w \in F(K)$. This completes the proof.
		\end{proof}
		\section{\textbf{Numerical Example}}
			\subsection*{Examples 4.1} Let $Z=\R$ and $W=[-1, 1] \subset Z$ equipped with the norm $\|x\|=|x|$. Define a self-mapping $K$ on $W$ by\\
			\begin{equation*}
				K(x)=1-x
			\end{equation*}
			Clearly the above mapping is a nonexpansive mapping.
			\\Let $\alpha_n = 0.2$ and $\beta_n = \frac{n}{2n+1} $ for all $n \in \mathbb{N}$ and taking $0.01$ and $-0.5$ as two different initial guesses. We obtained the following tables and graph of comparison of speed of convergence of some prominent iterative schemes.
			
			\begin{table}[H]
\small\addtolength{\tabcolsep}{-0.0000001pt}
\begin{center}
\scalebox{0.8}{
\begin{tabular}{|c|c|c|c|c|c|c|c|c|}
\hline
\multirow{2}{*}{\textbf{Steps}} & \multicolumn{4}{c|}{\textbf{Initial point = 0.01}} & \multicolumn{4}{c|}{\textbf{Initial point = -0.5}} \\
\cline{2-9}
 & Mann & Ishikawa & Noor & New & Mann & Ishikawa & Noor & New \\
\hline
1 & 0.01 & 0.01 & 0.01 & 0.01 & -0.5 & -0.5 & -0.5 & -0.5 \\
2 & 0.206 & 0.14067 & 0.17333 & 0.46733  & -0.1 & -0.23333 & -0.16667 & 0.43333\\
3 & 0.3236 & 0.22691 & 0.28658 & 0.50131 & 0.14 & -0.057333 & 0.064444 & 0.50267 \\
4 & 0.39416 & 0.28933 & 0.3628 & 0.49989 & 0.284 & 0.070057 & 0.22 & 0.49977 \\
5 & 0.4365 & 0.33614 & 0.4128 & 0.50001 & 0.3704 & 0.1656 & 0.32204 & 0.50003 \\
6 & 0.4619 & 0.37189 & 0.44504 & 0.5 & 0.42224 & 0.23856 & 0.38783 & 0.5 \\
7 & 0.47714 & 0.39949 & 0.46557 & 0.5 & 0.45334 & 0.29487 & 0.42974 & 0.5 \\
8 & 0.48628 & 0.42093 & 0.47854 & 0.5 & 0.47201 & 0.33863 & 0.45621 & 0.5 \\
9 & 0.49177 & 0.43767 & 0.48668 & 0.5 & 0.4832 & 0.3728 & 0.47281 & 0.5 \\
10 & 0.49506 & 0.4508 & 0.49175 & 0.5 & 0.48992 & 0.39958 & 0.48317 & 0.5 \\
11 & 0.49704 & 0.4611 & 0.49491 & 0.5 & 0.49395 & 0.42062 & 0.48961 & 0.5 \\
12 & 0.49822 & 0.46922 & 0.49686 & 0.5 & 0.49637 & 0.43719 & 0.4936 & 0.5 \\
13 & 0.49893 & 0.47562 & 0.49807 & 0.5 & 0.49782 & 0.45025 & 0.49607 & 0.5 \\
14 & 0.49936 & 0.48068 & 0.49882 & 0.5 & 0.49869 & 0.46057 & 0.49759 & 0.5 \\
15 & 0.49962 & 0.48468 & 0.49927 & 0.5 & 0.49922 & 0.46873 & 0.49852 & 0.5 \\
16 & 0.49977 & 0.48784 & 0.49956 & 0.5 & 0.49953 & 0.47518 & 0.49909 & 0.5 \\
17 & 0.47714 & 0.39949 & 0.46557 & 0.5 & 0.49972 & 0.4803 & 0.49945 & 0.5 \\
18 & 0.49992 & 0.49233 & 0.49983 & 0.5 & 0.49983 & 0.48435 & 0.49966 & 0.5 \\
19 & 0.49995 & 0.49391 & 0.4999 & 0.5 & 0.4999 & 0.48757 & 0.49979 & 0.5 \\
20 & 0.49997 & 0.49516 & 0.49994 & 0.5 & 0.49994 & 0.49012 & 0.49987 & 0.5 \\
\hline
\end{tabular}}
\end{center}
\caption{Iterative values.}
\end{table}

        \begin{figure}[!htbp]
        	\centering
        	\hspace*{1.0cm}
        	\includegraphics[width=1.19\textwidth]{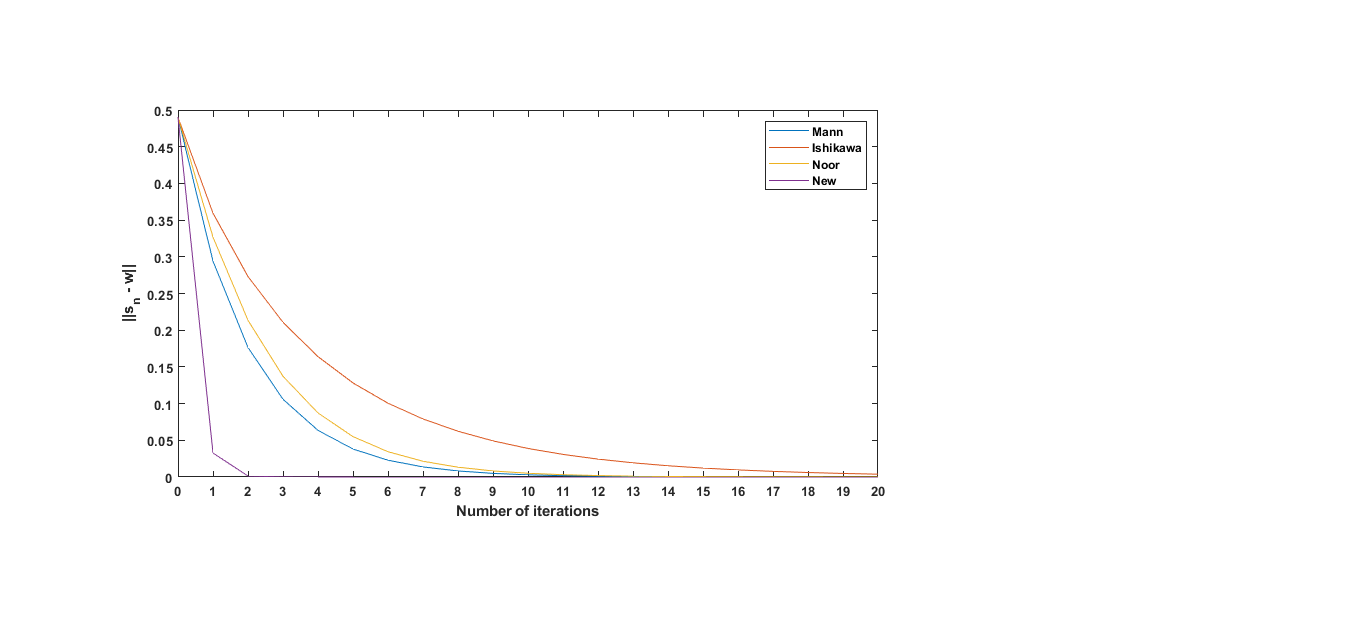} 
        	\vspace{-60pt}
            \centering
        	\caption{}
        	\label{fig:your-label}
        \end{figure}

        \section{Conclusion}
       In this paper, we proposed a new iterative scheme for approximating fixed points of nonlinear operators in a real Banach (or Hilbert) space. The theoretical analysis established both weak and strong convergence of the sequence generated by the proposed method under suitable conditions. Furthermore, a numerical example was presented to compare the performance of the new scheme with several existing iterative methods. The comparative results demonstrated that the proposed iterative process achieves a faster rate of convergence, highlighting its computational efficiency and practical relevance. These results indicate that the new scheme serves as a reliable and effective alternative for solving fixed point problems, and it may be extended in future work to broader classes of operators and applications.\\

        \hspace{-13pt}\textbf{Authors Contributions:}
        Both authors contributed equally.\\
        \textbf{Conflicts of Interest:} The authors declare no conflicts of interest.\\
        \textbf{Data Availability statement:} The Authors have nothing to report.\\
        \textbf{Ethics Statement:} The authors have nothing to report.\\
        \textbf{Funding:}No funding was received to assist with the preparation of this manuscript.

\end{document}